\documentclass[11pt,reqno,tbtags,a4paper]{amsart}
\usepackage{amssymb}
\usepackage{mathabx}  
\usepackage{xpunctuate}
\usepackage{url}
\usepackage[square,numbers]{natbib}
\bibpunct[, ]{[}{]}{;}{n}{,}{,}

\title[Uniform random variables and Eulerian numbers]
{
Better-than-average uniform random variables and Eulerian numbers, or:
How many candidates should a voter approve?
}

\date{February 20, 2024}

\author{Svante Janson}
\thanks{Supported by the Knut and Alice Wallenberg Foundation}
\address{Department of Mathematics, Uppsala University, PO Box 480,
SE-751~06 Uppsala, Sweden}
\email{svante.janson@math.uu.se}
\newcommand\urladdrx[1]{{\urladdr{\def~{{\tiny$\sim$}}#1}}}
\urladdrx{http://www2.math.uu.se/~svante/}

\author{Warren D. Smith}
\email{warren.wds@gmail.com}
\urladdrx{http://RangeVoting.org}


\numberwithin{equation}{section}

\renewcommand\le{\leqslant}
\renewcommand\ge{\geqslant}

\allowdisplaybreaks

\theoremstyle{plain}
\newtheorem{theorem}{Theorem}[section]
\newtheorem{lemma}[theorem]{Lemma}

\theoremstyle{definition}

\newcommand\xqed[1]{%
    \leavevmode\unskip\penalty9999 \hbox{}\nobreak\hfill
    \quad\hbox{#1}}

\newtheorem{exampleqqq}[theorem]{Example}
\newenvironment{example}{\begin{exampleqqq}}
  {\xqed{\qedsymbol}\end{exampleqqq}}

\newtheorem{remarkqqq}[theorem]{Remark}
\newenvironment{remark}{\begin{remarkqqq}}
  {\xqed{\qedsymbol}\end{remarkqqq}}

\newtheorem*{ack}{Acknowledgement}

\theoremstyle{remark}

\newenvironment{romenumerate}[1][-10pt]{
\addtolength{\leftmargini}{#1}\begin{enumerate}
 \renewcommand{\labelenumi}{\textup{(\roman{enumi})}}%
 }{\end{enumerate}}

\newcounter{oldenumi}
{\setcounter{oldenumi}{\value{enumi}}
\begin{romenumerate} \setcounter{enumi}{\value{oldenumi}}}
{\end{romenumerate}}

\newcounter{thmenumerate}

\newcounter{xenumerate}   

\newcommand{\refT}[1]{Theorem~\ref{#1}}

\newcommand{\refL}[1]{Lemma~\ref{#1}}

\newcommand{\refR}[1]{Remark~\ref{#1}}

\newcommand{\refS}[1]{Section~\ref{#1}}

\begingroup
  \divide\count255 by 60
  \multiply\count255 by -60
  \advance\count255 by \time
  \ifnum \count255 < 10 \xdef\klockan{\the\count1.0\the\count255}
  \else\xdef\klockan{\the\count1.\the\count255}\fi
\endgroup

\newcommand{\sumin}{\sum_{i=1}^n}

\newcommand\set[1]{\ensuremath{\{#1\}}}

\newcommand\biggset[1]{\ensuremath{\biggl\{#1\biggr\}}}

\newcommand\xpar[1]{(#1)}
\newcommand\bigpar[1]{\bigl(#1\bigr)}
\newcommand\Bigpar[1]{\Bigl(#1\Bigr)}

\newcommand\lrpar[1]{\left(#1\right)}

\newcommand\cpar[1]{\{#1\}}

\def\rompar(#1){\textup(#1\textup)}    

\def\xexp(#1){e^{#1}}
\newcommand\ceil[1]{\lceil#1\rceil}

\newcommand\floor[1]{\lfloor#1\rfloor}

\newcommand\ntoo{\ensuremath{{n\to\infty}}}

\newcommand\punkt{\xperiod}    
\newcommand\iid{i.i.d\punkt}    
\newcommand\eg{e.g\punkt}

\newcommand{\as}{a.s\punkt}

\newcommand\ii{\mathrm{i}}

\newcommand\eqd{\overset{\mathrm{d}}{=}}

\newcommand\bbR{\mathbb R}
\newcommand\bbC{\mathbb C}

\newcommand\bbZ{\mathbb Z}

\newcounter{CC}
\newcounter{cc}

\renewcommand\Re{\operatorname{Re}}
\renewcommand\Im{\operatorname{Im}}

\newcommand\E{\operatorname{\mathbb E}{}} 
\renewcommand\P{\operatorname{\mathbb P{}}}

\newcommand\Var{\operatorname{Var}}
\newcommand\Cov{\operatorname{Cov}}
\newcommand\Corr{\operatorname{Corr}}

\newcommand\Res{\operatorname{Res}}

\newcommand\ga{\alpha}

\newcommand\gD{\Delta}
\newcommand\gf{\varphi}

\renewcommand\phi{\xxx}  

\newcommand\cV{\mathcal V}

\newcommand\indic[1]{\boldsymbol1_{\cpar{#1}}} 

\newcommand\qww{^{-2}}

\newcommand\intoo{\int_0^\infty}
\newcommand\intoooo{\int_{-\infty}^\infty}
\newcommand\oi{\ensuremath{[0,1]}}

\newcommand\ooo{[0,\infty)}

\newcommand\dd{\,\mathrm{d}}

\newcommand{\mgf}{moment generating function}
\newcommand{\chf}{characteristic function}

\newcommand\rhs{right-hand side}

\newcommand\Uoi{U(0,1)}

\newcommand\xx[1]{_{(#1)}}

\newcommand\bU{\overline{U}}

\newcommand\fSni{\floor{S_{n-1}}}
\newcommand\cSni{\ceil{S_{n-1}}}
\newcommand{\euler}[2]{\genfrac{ < }{ > }{0pt}{}{#1}{#2}}
\newcommand\WW{\widehat W}
\newcommand\WWW{W^*}

\newcommand\arcsec{\operatorname{arcsec}}

\begin{document}

\begin{abstract} 
Consider $n$ independent random numbers with a uniform distribution on
$[0,1]$.
The number of them that exceed their mean is shown to have an Eulerian
distribution, i.e., it is described by the Eulerian numbers.
This is related to, but distinct from, the well known fact that the 
integer part of the sum of independent random numbers uniform on $[0,1]$
has an Eulerian distribution. 
One motivation for this problem comes from voting theory.
\end{abstract}

\maketitle

\section{Introduction}\label{S:intro}

Consider $n\ge2$ independent 
random numbers $U_1,\dots,U_n$, each with a uniform distribution on $\oi$.
Let $S_n:=\sumin U_i$ be their sum, and $\bU:=S_n/n$ their mean.
Finally, let
\begin{align}\label{wn}
W_n:=\#\set{i\le n: U_i > \bU},   
\end{align}
i.e, the number of the numbers $U_i$ that exceed their mean.
Since we assume $n\ge2$, a.s.\ (almost surely, i.e, with probability 1),
the numbers $U_1,\dots,U_n$ are distinct;
thus $\bU$ is strictly between the smallest and the largest $U_i$, and
consequently $1\le W_n\le n-1$.
What is the distribution of the random variable $W_n$?

Our original motivation for this question comes from voting theory, and is
briefly described in \refS{Svoting}.

Our main purpose is to prove the following theorem,
which shows that the distribution of $W_n$ 
is described by the \emph{Eulerian numbers} $\euler mk$
(see \refS{Sprel} for definitions and references);
we thus may say that $W_n$ has an \emph{Eulerian distribution}.
The proof is in \refS{Spf}.

\begin{theorem}\label{T1}
Define $W_n$ by \eqref{wn}, for some $n\ge2$.
Then, for every $k\ge0$,
\begin{align}\label{t1}
  \P(W_n=k)=\frac{1}{(n-1)!}\euler{n-1}{k-1}
.\end{align}
\end{theorem}
The Eulerian number in \eqref{t1} is non-zero if and only if
$1\le k\le n-1$, matching the observation above that $1\le W_n\le n-1$
a.s.

\begin{remark}\label{R<>}
Since we assume $n\ge2$, we have \as{}
 $U_i\neq\bU$ for every $i$; 
thus it does not make any difference if we replace 
``$>$'' in the definition \eqref{wn} by ``$\ge$''.
Furthermore,  by symmetry, we obtain the same distribution if we replace 
``$>$'' by ``$<$''  or ``$\le$''.
Also, of course instead of "uniform on $[0,1]$" for \eqref{wn} we could
have made the summands be uniform on any particular real interval $[a,b]$
with $a<b$.
\end{remark}

\refT{T1}
seems to be new, but we recall the following well known
result:

\begin{theorem}\label{T2}
  The integer part, i.e.\ floor, of the sum $S_n=\sumin U_i$ 
defined above has the Eulerian distribution
  \begin{align}\label{t2}
    \P\bigpar{\floor{S_n}=k}=\frac{1}{n!}\euler nk,
\qquad k\ge0.
  \end{align}
\end{theorem}
Note that $0\le\floor{S_n}\le n-1$ a.s.
The formula \eqref{t2} can also be interpreted as a formula for the volume
of the slices $\set{x\in\oi^n:k\le x_1+\dots + x_n<k+1}$ of the unit cube.

\begin{remark}
  The distribution of the sum $S_n$, 
i.e, the probability $\P(S_n\le x)$ for all real $x$,
was found already by Laplace
\cite[pp.~257--260]{Laplace},
see also \eg{} 
\citet[XI.7.20]{FellerI} 
and 
\cite[Theorem I.9.1]{FellerII}.
The result \eqref{t2}, which connects this distribution at integer points $x$
and the Eulerian numbers, was found much later; 
see \eg{} \cite{MeyerR} (implicitly), 
\cite{Tanny}, \cite[\S2]{FoataNATO},
\cite{StanleyNATO}, 
or \cite[Chapter 7]{Petersen}.
See also \cite{SJ280} for related results for non-integer $x$.
\end{remark}

Using \refT{T2}, we see that \refT{T1} is equivalent to the identity in
distribution
\begin{align}\label{ws}
  W_n-1 \eqd \floor{S_{n-1}},
\qquad n\ge2.
\end{align}
(Note the shift from $n$ to $n-1$.)
In fact, our proof of \refT{T1} in \refS{Spf} will prove \eqref{ws}
directly, then conclude \eqref{t1} from \refT{T2}.

\begin{ack}
  We thank Dominique Foata and Kyle Petersen
for helpful comments.
\end{ack}

\section{Preliminaries}\label{Sprel}

The Eulerian numbers $\euler nk$ were 
introduced by Euler \cite{E55,E212,E352},
in a method to sum infinite series, see also 
\cite{FoataEuler}.
They later were found to have many combinatorial interpretations.
One common definition is to define $\euler nk$ as the number of permutations
of length $n$ that have exactly $k$ \emph{descents},
see \eg{}
\cite[pp.~267--269]{CM},
\cite[pp.~21--23]{StanleyI},
\cite[\S26.14]{NIST},
\cite[\S1.3]{Petersen}.
They can also be defined by the generating function
\begin{align}
  \sum_{n,k=0}^\infty\euler nk x^k \frac{t^n}{n!} = \frac{1-x}{e^{(x-1)t}-x}.
\end{align}
or the recursion
\begin{align}
  \euler nk = (k+1)\euler{n-1}{k} + (n-k)\euler{n-1}{k-1}, 
\qquad n\ge1,\; k\in\bbZ,
\end{align}
with the initial values $\euler0k=\indic{k=0}$.
We use ``$\indic{c}$ to mean ``1 if $c$ is true, otherwise 0.''
For these and other properties of the Eulerian numbers,
see \eg{} the monograph \cite{Petersen}  
and the further references there.

\begin{remark}
Other notations are also used, in particular $A_{n,k}$. 
Moreover, some authors shift the index $k$ above by 1,
so that $A_{n,k}=\euler{n}{k-1}$. 
\end{remark}

\section{Proofs}\label{Spf}

We first consider a more general situation.

Let $U_1,\dots, U_n$ be any $n$ real numbers (random or not).
As above, let $S_n:=\sumin U_i$ be their sum, 
$\bU:=S_n/n$ their mean, 
and now define
\begin{align}\label{a1}
\WW_n:=\#\set{i\le n: U_i < \bU}.   
\end{align}
Furthermore, 
let $U\xx1\le\dots\le U\xx{n}$ be $U_1,\dots,U_n$ sorted into increasing order,
then let 
\begin{align}\label{jw1}
\gD_i:=U\xx{i+1}-U\xx{i}, 
\qquad 1\le i\le n-1,    
\end{align}
be the gaps between them.

Note that
\begin{align}\label{j1}
  U\xx{i}&=U\xx1+\sum_{j=1}^{i-1}\gD_j,
\qquad 1\le i\le n,
\end{align}
and thus
\begin{align}\label{j2}
S_n &= \sumin U\xx{i}
=nU\xx1+\sumin\sum_{j=1}^{i-1}\gD_j
=nU\xx1+\sum_{j=1}^{n-1}(n-j)\gD_j.
\end{align}

\begin{remark}\label{R<>2}
We use in this section $\WW_n$ rather than $W_n$ in \eqref{wn} for notational
convenience below; as said in \refR{R<>}, this makes no difference for our
original problem. In general, we can interchange $W_n$ and $\WW_n$ by
changing the sign on each $U_i$,
which reverses both the order and signs of
$U\xx1,\dots,U\xx{n}$, and thus changes $\gD_i$ to $\gD_{n-i}$, so the
increments are the same but in opposite order.
\end{remark}

\begin{lemma}\label{Lgen}
  For any real numbers $U_1,\dots,U_n$, we have, for $k=0,\dots,n-1$,
  \begin{align}\label{j3}
\WW_n \le k \iff     
\sum_{j=1}^{k}j\gD_j
-
\sum_{j=k+1}^{n-1}(n-j)\gD_j
\ge0
.\end{align}
\end{lemma}

\begin{proof}
By \eqref{a1} and \eqref{j1}--\eqref{j2},
\begin{align}\label{j4}
  \WW_n\le k&
\iff
U\xx{k+1}\ge\bU
\iff
nU\xx{k+1}\ge S_n
\iff
n\sum_{j=1}^{k}\gD_j
\ge
\sum_{j=1}^{n-1}(n-j)\gD_j
\notag\\&
\iff
\sum_{j=1}^{k}n\gD_j
-
\sum_{j=1}^{n-1}(n-j)\gD_j
\ge0
.\end{align}
Furthermore,
\begin{align}\label{j5}
  \sum_{j=1}^{k}n\gD_j
-
\sum_{j=1}^{n-1}(n-j)\gD_j
=
\sum_{j=1}^{k}j\gD_j
-
\sum_{j=k+1}^{n-1}(n-j)\gD_j,
\end{align}
and thus \eqref{j3} follows.
\end{proof}

\begin{remark}
  In other words,
  \begin{align}
    \WW_n = 
\min\biggset{k\in\set{0,\dots,n-1}
  :\sum_{j=1}^kj\gD_j \ge\sum_{j=k+1}^{n-1}(n-j)\gD_j
}
. \end{align}
\end{remark}

\begin{proof}[Proof of \refT{T1}]
We use the notation above.
  Now $U_1,\dots,U_n$ are independent uniform variables in $\oi$, and thus
  $U\xx1,\dots,U\xx{n}$ are the corresponding order statistics. 
We define also $U\xx0=0$ and $\gD_0:=U\xx1$. (Thus \eqref{jw1} is now also
for $i=0$.)

The  vector $(U\xx1,\dots,U\xx{n})$ is uniformly distributed in the simplex
$\set{0\le u_1\le\dots\le u_n\le1}$. It follows that
$U\xx1/U\xx{n},\dots,U\xx{n-1}/U\xx{n}$ have the same joint distribution as
$n-1$ independent uniform random $\Uoi$ variables, arranged in increasing order.
In particular, using \eqref{j1} and \eqref{j2} (with $n-1$ instead of $n$),
\begin{align}\label{a7}
  S_{n-1}\eqd \sum_{i=1}^{n-1}\frac{U\xx{i}}{U\xx{n}}
=\frac{\sum_{i=1}^{n-1}U\xx{i}}{U\xx{n}}
=\frac{\sum_{j=0}^{n-2}(n-1-j)\gD_j}{\sum_{j=0}^{n-1}\gD_j}
\end{align}
and thus, for $k=0\dots,n-1$,
\begin{align}\label{j8}
\cSni \le k &
\iff
S_{n-1}\le  k
\iff
\sum_{j=0}^{n-1}(n-1-j)\gD_j\le k\sum_{j=0}^{n-1}\gD_j
\notag\\&
\iff 
\sum_{j=0}^{n-1}(j+k+1-n)\gD_j\ge0.
\notag\\&
\iff 
\sum_{i=1}^{k}i\gD_{i+n-1-k}
+\sum_{i=k+1}^{n-1}(i-n)\gD_{i-1-k}\ge0
.\end{align}
Comparing this and \eqref{j3}, we see that the conditions are the same up to
a (cyclic) permutation of the variables $\gD_0,\dots,\gD_{n-1}$.
It is well known, and easy to see, 
that for \iid{} uniform random variables $U_1,\dots,U_n$,
the increments $\gD_0,\dots,\gD_{n-1}$ have an 
\emph{exchangeable} distribution;
thus we obtain from \eqref{j3} and \eqref{j8}, for $k=0,\dots,n-1$,
\begin{align}\label{j9}
\P\bigpar{\WW_n \le k}&
=\P\lrpar{
\sum_{j=1}^{k}j\gD_j
-
\sum_{j=k+1}^{n-1}(n-j)\gD_j
\ge0}
\notag\\&
=\P\lrpar{\sum_{i=1}^{k}i\gD_{i+n-1-k}
+\sum_{i=k+1}^{n-1}(i-n)\gD_{i-1-k}\ge0}
\notag\\&
=\P\bigpar{\cSni \le k}
.\end{align}
Consequently, $\WW_n\eqd\cSni$.
Finally, $W_n\eqd \WW_n$ by \refR{R<>}, and
$\cSni=\fSni+1$ almost surely.
This proves the identity in distribution $W_n\eqd \fSni+1$,
which as remarked in the introduction is equivalent to \refT{T1} by the well
known \refT{T2}.
\end{proof}

We note an alternative proof which calculates the probability of the event
in \eqref{j3} directly, using properties of exponential random variables.
While we regard this proof as more complicated than the proof above, it
might be of independent interest.

\begin{proof}[Alternative proof of \refT{T1}]
As in the proof above, let $U_1,\dots,U_n$ be \iid{} uniform random
variables in $\oi$, and let the order statistics $U\xx i$ and increments
$\gD_i$ be as above.

Let $\xi_1,\xi_2,\dots$ be \iid{} random variables with a standard
exponential density $\xexp(-x)$ for $x>0$.
Let $\zeta_1,\zeta_2,\dots$ be the partial sums $\zeta_k:=\sum_{i=1}^k \xi_i$.
(These are the points of a rate 1 Poisson point process on $\ooo$.)
It is well known that conditioned on $\zeta_{n+1}$,
the points
$\zeta_1,\dots,\zeta_n$ are  distributed as $n$ independent uniformly
distributed points on $[0,\zeta_{n+1}]$ that have been ordered, and thus
$(\zeta_1/\zeta_{n+1},\dots,\zeta_n/\zeta_{n+1})\eqd(U\xx1,\dots,U\xx{n})$.
Hence $(\gD_0,\dots,\gD_{n-1})\eqd(\xi_1/\zeta_{n+1},\dots,\xi_n/\zeta_{n+1})$,
and it follows by \refL{Lgen} and \refR{R<>} that, for $k=0,\dots,n-1$,
\begin{align}\label{w1}
  \P(W_n\le k)&
=\P\lrpar{\sum_{j=1}^{k}j\gD_j-\sum_{j=k+1}^{n-1}(n-j)\gD_j \ge0}
\notag\\&
=\P\lrpar{\sum_{j=1}^{k}j\frac{\xi_{j+1}}{\zeta_{n+1}}
-\sum_{j=k+1}^{n-1}(n-j)\frac{\xi_{n-j+1}}{\zeta_{n+1}} \ge0}
\notag\\&
=\P\lrpar{\sum_{j=1}^{k}j{\xi_{j+1}}
-\sum_{j=k+1}^{n-1}(n-j){\xi_{n-j+1}} \ge0}
.\end{align}
The last probability is calculated by \refL{LP}, with $\ell=n-k-1$,
$a_i=i$ and $b_j=j$ (after reordering), and  thus, by \eqref{lp},
\begin{align}\label{w2}
  \P(W_n\le k)&
=\sum_{m=1}^k m^{n-2}\frac{(-1)^{k-m}m!}{(m-1)!\,(k-m)!\,(m+n-k-1)!}
\notag\\&
=\sum_{m=1}^k m^{n-1}\frac{(-1)^{k-m}}{(k-m)!\,(n-1+m-k)!}
\notag\\&
=\sum_{i=0}^{k-1} (k-i)^{n-1}\frac{(-1)^{i}}{i!\,(n-1-i)!}
\notag\\&
=\frac{1}{(n-1)!}\sum_{i=0}^{k-1} (-1)^{i}\binom{n-1}{i}(k-i)^{n-1}
.\end{align}
Hence, by a simple calculation,
\begin{align}\label{w3}
&  \P(W_n= k)
=   \P(W_n\le k)-   \P(W_n\le k-1)
\notag\\&
=\frac{1}{(n-1)!}\sum_{i=0}^{k-1} (-1)^{i}\binom{n-1}{i}(k-i)^{n-1}
-\frac{1}{(n-1)!}\sum_{i=0}^{k-2} (-1)^{i}\binom{n-1}{i}(k-1-i)^{n-1}
\notag\\&
=\frac{1}{(n-1)!}\sum_{i=0}^{k-1} (-1)^{i}\binom{n}{i}(k-i)^{n-1}
=\frac{1}{(n-1)!}\euler{n-1}{k-1},
\end{align}
where the final inequality is a standard formula for Eulerian numbers, see
\eg{}
\cite[(6.38)]{CM},
\cite[Corollary 1.3]{Petersen}, or 
\cite[26.14.6]{NIST}.

In fact, the \rhs{} of \eqref{w2}  
equals the probability $\P(\ceil{S_{n-1}}\le k)=\P({S_{n-1}}\le k)$
as calculated 
in \eg{} \cite{Laplace} and \cite{FellerII}, and \eqref{w3}
just repeats a calculation in \eg{} \cite{Tanny} or \cite{FoataNATO}.
(We may also use \eqref{j8} and the method above to 
obtain a possibly new proof of this formula for $\P(\ceil{S_{n-1}}\le k)$ 
 and thus of \refT{T2}.)
\end{proof}

\section{Application in (or motivation from) voting theory}\label{Svoting}

One motivation for the study of the random variable $W_n$ comes from voting theory.

Consider an election (of 1 person)
by \emph{Approval Voting} \cite{BramsF}. This means that each voter
\emph{approves} (i.e, votes for) an arbitrary set of candidates, 
and the candidate approved by the most voters wins.
(Ties are broken randomly.)

We are interested in the best strategy of a single voter.
We assume each voter has evaluated all $n$ candidates and therefore assigned
a personal
real ``utility'' $U_i$ to each candidate $i=1,\dots,n$.
The voter does not know the preferences (or strategies) of the other voters,
so from the perspective of this voter, the other votes can be regarded as
random, with some distribution.
We assume also that the voter tries to maximize the expected utility of the 
winner. 

Let us further assume that the voter is completely ignorant of the choices
of other voters, so the assumed probability distribution of the aggregated
votes of all others is symmetric under the group of $n!$ permutations of the $n$ candidates.
Since ties are broken randomly we may simplify the analysis by assuming
that this is done by giving each candidate a (secret) random number of extra
votes $X_i\in(0,1)$ before the election, with $X_1,\dots,X_n$ \iid{} and
uniformly distributed. 
We include these extra votes in our probability model of the
other voters, thus (a.s.)\ eliminating ties.
Let $V_i$ be the number of votes from all other voters 
for candidate $i$, including the $X_i$ extra ones. 
(Thus $V_i$ is a real number.)
The assumed probability distribution of $V_1,\dots,V_n$ is still symmetric.
Hence there is a (presumably small) probability $p$
such that for each pair $(i,j)$,\ 
$V_i$ is the largest
number of votes, \emph{but} $V_i-V_j <1$
so that our voter could swing the outcome.
Let us also assume that the probability that there is more than one such
``swing pair'' $(i,j)$ is neglectibly small -- which is
reasonable if the number of voters is large.
Assume our voter approves a subset $A$ of the candidates.
Then let $I_i$ denote the indicator-function $\indic{i\in A}$ that candidate $i$ is approved.
Then the expected change that this vote 
causes to the utility of the winner is, using the notations $S_n$ and $\bU$
as in \refS{Spf},
\begin{align}\label{av1}
  \sum_{i\neq j}pI_j(1-I_i)(U_j-U_i)
&=
 p \sum_{i,j=1}^nI_j(1-I_i)(U_j-U_i)
= p \sum_{i,j=1}^nI_j(U_j-U_i)
\notag\\&
= p n\sum_{j=1}^nI_jU_j
- p \sum_{j=1}^nI_jS_n
= p n\sum_{j=1}^nI_j(U_j-\bU).
\end{align}
Since the objective is to maximize the expected utility of the winner, 
or equivalently to maximize \eqref{av1}, the optimal strategy is to 
choose $I_j:=\indic{U_j>\bU}$; in other words:
$$\text{\vbox{
\emph{The voter should approve each candidate with a utility exceeding the
average utility
(all utilities as reckoned by that voter)}.
}}$$
Note that this argument
does not depend on any assumptions (beyond its symmetry) about the shape of the
probability distribution for the other voters, 
Hence this strategy is optimal if we are ignorant about the other voters;
and we do not have to construct any specific probability distribution describing
their voting.

Consider now a number of voters such that each of them follows this
optimal strategy, where the voters make their own assignments of utilities.
A simple model is that each voter’s utilities are $n$ i.i.d.\ random numbers uniform in $\oi$.
Then the number of candidates that the voter approves is given by $W_n$ in
\eqref{wn}, whereupon \refT{T1} shows that
$$\text{
\emph{The number of approved candidates has an Eulerian distribution}.
}$$

A more realistic model than ``total ignorance about the other voters''
might be to assume that there is a certain ``good'' subset of the candidates who have
a reasonable chance to win, and whom almost all voters regard as better than the complement ``bad''
candidate-subset.  If so, then we'd predict that the number of approved candidates has
distribution approximately Eulerian among the good-subset only.

Also, much more generally:  a voter \emph{not} wholly ignorant
of all the other voters, 
typically\footnote{
Actually this strategy is not always optimal.
A counterexample:
Suppose there are three candidates $A$, $B$, $C$.
My utilities are 10, 6, 0.
The other $2m$ voters
are very polarized, with exactly half approving
$A$ but not $B$, and half approving $B$ but not $A$.
But they all toss fair coins
to decide whether to approve $C$.
Then $A$, $B$, $C$'s win-chances are (in the $m\to\infty$ limit)
$1/4$, $1/4$, $1/2$, so my
expected utility is 4.
Hence the recommended strategy approves ${A,B}$,
yielding expected utility still 4.
But here a better strategy is approving only $A$,
yielding win probabilities $1/2$, $0$, $1/2$
so that my expected utility is 5.
More of this sort of thing: \cite{RVstrat2},\cite{GibbSatComplete}.
}
should approve all candidates whose utility (to her)
exceeds the expected utility (to her) of the winner. 
To determine the latter,  the voter needs
to estimate the win-probabilities for all
the candidates; and the more knowledge she has about how the other voters
are likely to behave, the better such estimates she can make.
We have discussed above only the special case where the voter has zero knowledge about
the other voters; but many other special cases also can be treated, including with the aid of
\refL{LP}, e.g. see \cite{RangeVoting}.

It is interesting to study the distribution of the
random variable $W_n$ also for other
distributions of $(U_1,\dots,U_n)$. We give here a quite different case
where the distribution can be found explicitly.

\begin{example}
We may start with any positive random variables $\gD_1,\dots,\gD_{n-1}$.
We then define $U\xx1,\dots,U\xx{n}$ by \eqref{j1}, with, say, $U\xx1=0$,
and finally let $U_1,\dots,U_n$ be a random permutation of
$U\xx1,\dots,U\xx{n}$. This yields random utilities with the increments
$\gD_1,\dots,\gD_{n-1}$.

Consider the case when the increments $\gD_i$ are
\iid{} random variables with a positive (strictly)
stable distribution of index $\ga\in(0,1)$;
thus their Laplace transform is $\E e^{-t\gD_i}=e^{-ct^\ga}$ for $t\ge0$ and
a constant $c>0$ (which is unimportant by rescaling).
We find the distribution of $W_n$ by using \refL{Lgen}, noting that
$W_n\eqd\WW_n$ by the sign-change argument in \refR{R<>2}. 
In this case
\begin{align}\label{er1}
  \sum_{j=1}^k j\gD_j \eqd s_k^{1/\ga} \gD_1,
\end{align}
where
\begin{align}\label{er2}
  s_k:=\sum_{j=1}^k j^\ga.
\end{align}
Similarly,
\begin{align}\label{er3}
\sum_{j=k+1}^{n-1}(n-j)\gD_j
= \sum_{i=1}^{n-k-1}i\gD_{n-i}
\eqd s_{n-k-1}^{1/\ga} \gD_1.
\end{align}
Furthermore, the two sums in \eqref{j3} are independent, 
and it follows that, for $0\le k\le n-1$,
\begin{align}\label{er4}
  \P\bigpar{W_n\le k}
=\P\bigpar{s_{k}^{1/\ga} \gD_1 \ge s_{n-k-1}^{1/\ga} \gD_2}
=\P\lrpar{\frac{\gD_2}{ \gD_1} \le \Bigpar{\frac{s_{k}}{s_{n-k-1}}}^{1/\ga} }
.\end{align}

In the special case $\ga=1/2$ (when the distribution of $\gD_i$ is known as
the L\'evy distribution), we can take $\gD_i=Z_i\qww$, where $Z_1,Z_2,\dots$
are \iid{} standard normal random variables.
We thus obtain
\begin{align}\label{er5}
  \P\bigpar{W_n\le k}&
=\P\lrpar{\Bigpar{\frac{Z_2}{Z_1}}\qww 
\le \Bigpar{\frac{s_{k}}{s_{n-k-1}}}^2 }
=\P\lrpar{\frac{|Z_1|}{|Z_2|} \le\frac{s_{k}}{s_{n-k-1}}}
\notag\\&
=\frac{2}{\pi}\arctan\frac{s_{k}}{s_{n-k-1}}
=\frac{2}{\pi}\arctan\frac{\sum_{j=1}^k\sqrt{j}}{\sum_{j=1}^{n-k-1}\sqrt{j}}
,\end{align}
since the argument of the complex number $|Z_2|+\ii|Z_1|$ is uniformly
distributed on $[0,\frac{\pi}2]$.

We may note that it follows from \eqref{er5} that as \ntoo,  $W_n/n$
converges in distribution to a random variable $\WWW$ with
\begin{align}
  \P(\WWW\le x) = \frac{2}{\pi}\arctan\Bigpar{\frac{x}{1-x}}^{3/2},
\qquad 0\le x \le1,
\end{align}
and thus density function
\begin{align}
  f_{\WWW}(x) =\frac{3}{\pi}\cdot \frac{\sqrt{x(1-x)}}{x^3+(1-x)^3}
,\qquad 0<x<1.
\end{align}
\end{example}

\begin{example}
  Another interesting case is when the numbers $U_1,\dots,U_n$ are
  \iid{} standard normal random variables.
Let $Y_i:=U_i-\bU$; then $(Y_1,\dots,Y_n)$ has a normal distribution with
mean 0 and covariances $\Var(Y_i)=\frac{n-1}n$,
$\Cov(Y_i,Y_j)=-\frac{1}{n}$, $i\neq j$, and thus correlations 
$\Corr(Y_i,Y_j)=-\frac{1}{n-1}$, $i\neq j$.
Note that the vector $(Y_1,\dots,Y_n)$ lies in the $(n-1)$-dimensional
subspace $\cV:=\set{(x_i)_1^n:\sum_1^n x_i=0}$ of $\bbR^n$.

Let again $W_n$ be given by \eqref{wn}. By symmetry, for $1\le k\le n-1$,
\begin{align}\label{feb1}
  \P(W_n=k)=\binom nk \P\bigpar{Y_1,\dots,Y_k>0 > Y_{k+1},\dots,Y_n}.
\end{align}
For each $i=1,\dots n$, the intersection of the hyperplane $x_i=0$ with
$\cV$ defines a hyperplane $H_i$ in $\cV$. Similarly, we define a half-space
$A_i$ in $\cV$ as $\cV$ intersected with $\set{x_i>0}$ if $1\le i\le k$, 
and $\set{x_i<0}$ if $k+1\le i\le n$. The intersection 
$C_k:=\bigcap_{i=1}^n A_i$ of these half-spaces is a polyhedral cone,
bounded by (parts of) the hyperplanes $H_i$. 
Color these hyperplanes \emph{red} for $1\le i\le k$ and \emph{blue}
for $k+1\le i \le n$. 
The interior angle between $H_i$ and $H_j$ ($i\neq j$) is 
$\arcsec(n-1)$ if $i$ and $j$ have the same color
and $\arcsec(1-n)$ otherwise.

By intersecting the cone $C_k$ and the $(n-2)$-dimensional
unit sphere in $\cV$, we obtain a spherical polyhedron $T_k$, and if $f(n,k)$ is
its volume, then \eqref{feb1} yields
\begin{align}
  P(W_n=k)=\binom nk \frac{f(n,k)}{\Omega(n-1)}
=\binom nk \frac{f(n,k)}{2 \pi^{(n-1)/2}/\Gamma((n-1)/2)}
.\end{align}
where $\Omega(n-1)$ is the surface area of the unit sphere in $R^{n-1}$.

For $n=4$ and $k=1$,  
$T_k$ is a spherical triangle with all three angles 
$\arcsec(3)$, and thus area $f(4,1)=3 \arcsec(3)-\pi$, yielding
\begin{align}
  \P(W_4=1)=4\frac{3\arcsec(3)-\pi}{4\pi}
=\frac{3\arcsec(3)}{\pi}-1
=\frac{\arccos(23/27)}{\pi}
\doteq 0.175479656
\end{align}
By symmetry, $\P(W_4=3)=\P(W_4=1)\doteq 0.175479656$, and thus
$\P(W_4=2)\doteq 1-2\P(W_4=1)=3-6\arcsec(3)/\pi
\doteq 0.649040689.
$
These results agree excellently with Monte Carlo experiments.
This distribution, not surprisingly, differs from the Eulerian
distribution found in \refT{T1} for \emph{uniform} random variables.

For $n=5$, we similarly can obtain
$\P(W_5=1)=\P(W_5=4)$ from the volume of a
spherical tetrahedron with
all face-pair dihedral angles $\arcsec(4)$.
Computing this volume using Murakami's formula
\cite{Murakami}, we find
$0.19314200684738698696896128783\dots$ 
yielding
$\P(W_5=1)=\P(W_5=4)=0.04892344186208439057262451667\dots$ 
agreeing
with $0.048923 \pm 0.000003$ from Monte Carlo.
(Warning: some implementations of Murakami's formula
trigger bugs in Mathematica${}^{\rm TM}$; and
Murakami's formula has severe numerical problems
for regular tetrahedra with dihedral angle
$\pi-\epsilon$ for small enough $\epsilon> 0$.)
Surprisingly, we noticed that this volume apparently equals
$\arccos(1-3/4^3)\pi/5$, which Murakami [personal communication]
then confirmed to hundreds of decimal places accuracy -- but
it remains unproven.  If so, then
$\P(W_5=1)=\P(W_5=4)=\arccos(61/64)/(2\pi)$.
\end{example}

\appendix
\section{A probability for exponential variables}\label{App}

We give here a calculation for linear combinations of
exponential variables used in the alternative proof in \refS{Spf}.
We give it in a general form. The assumption that $a_1,\dots,a_\ell$ are
distinct is necessary for the formula as stated here to make sense; however, 
the proof extends to cases with repeated $a_i$, but then the residue
calculations will be more complicated.

\begin{lemma}\label{LP}
  Let $\xi_1,\xi_2,\dots$ and $\eta_1,\eta_2,\dots$ be \iid{} standard-exponential distributed random
  variables, i.e. probability density $exp(-x)$
Let $a_1,\dots,a_k$ and $b_1,\dots,b_\ell$ be positive real numbers, with
$k,\ell\ge0$, $k+\ell\ge1$, and
$a_1,\dots,a_k$ distinct.
Then
\begin{align}\label{lp}
  \P\lrpar{\sum_{i=1}^k a_i\xi_i \ge \sum_{j=1}^\ell b_j\eta_j}
=\sum_{m=1}^k \frac{a_m^{k+\ell-1}}
{\prod_{i\neq m}\xpar{a_m-a_i}\prod_{j=1}^\ell\xpar{a_m+b_j}}
.\end{align}
\end{lemma}

\begin{proof}
The cases $(k,\ell)=(0,1)$ or $(1,0)$ are trivial, so we may assume
$k+\ell\ge2$. 

  Let $X:=\sum_{i=1}^k a_i\xi_i - \sum_{j=1}^\ell b_j\eta_j$.
Then $X$ has the \mgf{}
\begin{align}\label{Psi}
  \E e^{zX} = \prod_{i=1}^k\frac{1}{1-a_iz}\prod_{j=1}^\ell\frac{1}{1+b_jz}
=:\Psi(z),
\end{align}
for all complex $z$ with $|\Re z|$ sufficiently small.
In particular, the \chf{} of $X$ is $\gf(t)=\Psi(\ii t)$, $-\infty<t<\infty$.

We let \eqref{Psi} define $\Psi(z)$ in the entire complex plane $\bbC$,
and note that $\Psi(z)$ is a  meromorphic function with poles at $1/a_i$ and
$-1/b_j$. We have the estimates, for some constants $C$ and $A$,
\begin{align}\label{Psi1}
  |\Psi(z)|&\le C |\Im z|^{-k-\ell} , &&\hskip-6em z\in\bbC,
\\\label{Psi2}
  |\Psi(z)|&\le C|z|^{-k-\ell} ,  &&\hskip-6em |z|\ge A
.\end{align}

Since $k+\ell\ge2$, \eqref{Psi1} implies that 
the \chf{} is integrable, and $X$ has a continuous
density function given by
\begin{align}\label{aam1}
  f(x) = \frac{1}{2\pi}\intoooo e^{-\ii xt}\gf(t)\dd t
= \frac{1}{2\pi\ii}\int_{-\infty\ii}^{\infty\ii} e^{- xz}\Psi(z)\dd z,
\qquad x\in\bbR,
\end{align}
where the last integral is along the imaginary axis.

Let $x>0$. We move the line of integration to $\Re z= r$, using the estimate
\eqref{Psi1} and Cauchy's residue theorem; we then let $r\to\infty$, which
makes the integral tend to 0 by \eqref{Psi1}--\eqref{Psi2}, and obtain
\begin{align}\label{aam2}
  f(x)=
-\sum_{m=1}^k \Res_{z=1/a_m}\bigpar{e^{-xz}\Psi(z)}
= \sum_{m=1}^k e^{-x/a_m}\frac{1}{a_m}
\prod_{i\neq m}\frac{1}{1-a_i/a_m}\prod_{j=1}^\ell\frac{1}{1+b_j/a_m}.
\end{align}
(This is where we use the assumption that $a_1,\dots,a_m$ are distinct,
which means that the poles $1/a_m$ are simple.)
Integrating \eqref{aam2} yields
\begin{align}\label{aam3}
\P(X\ge0)=\intoo  f(x)\dd x
= \sum_{m=1}^k 
\prod_{i\neq m}\frac{1}{1-a_i/a_m}\prod_{j=1}^\ell\frac{1}{1+b_j/a_m},
\end{align}
which is equivalent to \eqref{lp}.
\end{proof}

\newcommand\AAP{\emph{Adv. Appl. Probab.} }
\newcommand\JAP{\emph{J. Appl. Probab.} }
\newcommand\JAMS{\emph{J. \AMS} }
\newcommand\MAMS{\emph{Memoirs \AMS} }
\newcommand\PAMS{\emph{Proc. \AMS} }
\newcommand\TAMS{\emph{Trans. \AMS} }
\newcommand\AnnMS{\emph{Ann. Math. Statist.} }
\newcommand\AnnPr{\emph{Ann. Probab.} }
\newcommand\CPC{\emph{Combin. Probab. Comput.} }
\newcommand\JMAA{\emph{J. Math. Anal. Appl.} }
\newcommand\RSA{\emph{Random Structures Algorithms} }
\newcommand\DMTCS{\jour{Discr. Math. Theor. Comput. Sci.} }

\newcommand\AMS{Amer. Math. Soc.}
\newcommand\Springer{Springer-Verlag}
\newcommand\Wiley{Wiley}

\newcommand\vol{\textbf}
\newcommand\jour{\emph}
\newcommand\book{\emph}
\newcommand\inbook{\emph}
\def\no#1#2,{\unskip#2, no. #1,} 
\newcommand\toappear{\unskip, to appear}

\newcommand\arxiv[1]{\texttt{arXiv}:#1}
\newcommand\arXiv{\arxiv}

\newcommand\xand{and }
\renewcommand\xand{\& }

\def\nobibitem#1\par{}

\end{document}